\documentclass[12pt,a4paper]{amsart}
 
\usepackage[euler-digits]{eulervm}

\usepackage{amsfonts, amsthm, amssymb, amsmath}
\usepackage{mathrsfs, array, graphicx,xypic}
\usepackage{eucal,fullpage,times,color,enumerate,accents,paralist}
\usepackage{url}
\usepackage{float}
\usepackage{pbox}
\usepackage{xcolor,mathtools}
\usepackage{ytableau}
\usepackage{color}
\usepackage{mathrsfs}
\usepackage{amssymb}
\usepackage{tikz}
\usepackage{tikz-cd}
\usepackage{bm}
\usepackage{enumerate}
\usetikzlibrary{calc}
\graphicspath{ {images/} }

\DeclareMathAlphabet{\mathpzc}{OT1}{pzc}{m}{it}
\makeatletter
\def\blfootnote{\xdef\@thefnmark{}\@footnotetext}
\makeatother

\address{{\bf Dori Bejleri}\newline Mathematics Department, Brown University}
\email{dbejleri@math.brown.edu} 

\address{{\bf Gjergji Zaimi}}
\email{gjergjiz@gmail.com}

\title{The Topology of Equivariant Hilbert Schemes}
\author{{\larger D}{\smaller ori} {\larger B}{\smaller ejleri}\ \& \ {\larger G}{\smaller jergji} {\larger Z}{\smaller aimi}}
\date{\today}

\usepackage[backref,linktoc=page]{hyperref}
\hypersetup{
  colorlinks   = true,          
  urlcolor     = purple,          
  linkcolor    = blue,          
  citecolor   = purple             
}

\newtheorem*{mainthma}{Main Theorem A}
\newtheorem*{mainthmb}{Main Theorem B}

\newtheorem*{prop1.1}{Proposition 1.1}

\newtheorem*{thm1.2}{Theorem 1.2}
\newtheorem*{thm1.3}{Theorem 1.3}
\newtheorem{thm}{Theorem}[section]
\newtheorem{prop}{Proposition}[section]
\newtheorem{cor}{Corollary}[section]

\newtheorem{lemma}{Lemma}[section]

\newtheorem{defn}{Definition}[section]
\newtheorem{rem}{Remark}[section]

\newtheorem{ques}{Question}[section]

\newcommand{\mb}[1]{\mathbb{#1}}
\newcommand{\mf}[1]{\mathfrak{#1}}
\newcommand{\mc}[1]{\mathcal{#1}}

\newcommand{\Hom}{\operatorname{Hom}}

\newcommand{\Hilb}{\operatorname{Hilb}}
\newcommand{\Sym}{\operatorname{Sym}}

\def\ZZ{{\mathbb Z}}
\def\QQ{{\mathbb Q}}

\def\CC{{\mathbb C}}

\def\OO{{\mathcal O}}

\begin{document}
\begin{abstract} For $G$ a finite group acting linearly on $\mb{A}^2$, the equivariant Hilbert scheme $\Hilb^r[\mb{A}^2/G]$ is a natural resolution of singularities of $\Sym^r(\mb{A}^2/G)$. In this paper we study the topology of $\Hilb^r[\mb{A}^2/G]$ for abelian $G$ and how it depends on the group $G$. We prove that the topological invariants of $\Hilb^r[\mb{A}^2/G]$ are periodic or quasipolynomial in the order of the group $G$ as $G$ varies over certain families of abelian subgroups of $GL_2$. This is done by using the Bialynicki-Birula decomposition to compute topological invariants in terms of the combinatorics of a certain set of partitions.  
\end{abstract}
\maketitle
\tableofcontents
\section{Introduction}

Let $X$ be a smooth algebraic surface carrying the action of a finite group $G$. The \textit{equivariant Hilbert scheme} $\Hilb^r[X/G]$ (Section \ref{sec:intro}) is a generalization of the Hilbert scheme of points on $X$ that parametrizes certain $G$-equivariant subschemes. It is a natural resolution of singularities for the symmetric product $\Sym^r(X/G)$ of the quotient space. In this paper we study how the topology of these Hilbert schemes change as the group $G$ varies. 

When $G$ is an abelian group acting linearly on $X = \mb{A}^2$, we exhibit (Main Theorems \hyperref[mainthma]{A} and \hyperref[mainthmb]{B}) periodicity and quasipolynomiality for the Betti numbers and Euler characteristics of $\Hilb^r[\mb{A}^2/G]$ as the order of the group $G$ varies within certain familes of finite abelian subgroups of $GL_2$. The main tool is the combinatorics of \textit{balanced partitions} (Section \ref{sec:balanced}) and the proof is mostly combinatorial. To our knowledge, there is a priori no geometric relationship between the equivariant Hilbert schemes for the different groups we consider and it is an interesting question to understand why one might expect these results. 

\subsection{Statement of main results}
\label{sec:intro}

Let $G$ be a finite subgroup of $GL_2$. The stack quotient $[\mb{A}^2/G]$ of $\mb{A}^2$ by the action of $G$ is a smooth two dimensional orbifold with singular coarse moduli space $\mb{A}^2/G$. The Hilbert scheme of points $\Hilb^r[\mb{A}^2/G]$ is a $2r$-dimensional quasiprojective scheme parametrizing flat families of substacks of $[\mb{A}^2/G]$ with constant Hilbert polynomial $r$ \cite[Theorem 1.5]{quot}. Equivalently, $\Hilb^r[\mb{A}^2/G]$ is the moduli space of $G$-equivariant ideals $I \subset \mb{C}[x,y]$ such that $\mb{C}[x,y]/I \cong \mb{C}[G]^r$ as representations \cite[Proposition 2.9]{lili}. It is a union of irreducible components of the fixed locus $(\Hilb^{r|G|}(\mb{A}^2))^G$ \cite[Proposition 4.1]{brion}. In fact $\Hilb^r[\mb{A}^2/G]$ is smooth (Section \ref{sec:cotangent}).

There is a Hilbert-Chow morphism 

$$
\Hilb^r[\mb{A}^2/G] \to \Sym^r(\mb{A}^2/G)
$$

\noindent sending an ideal to its support in the coarse moduli space. The restriction of this morphism to the component of $\Hilb^r[\mb{A}^2/G]$ containing the locus of $r$ distinct free $G$-orbits is a resolution of singularities \footnote{When $G$ is abelian, $\Hilb^r[\mb{A}^2/G]$ is connected (Corollary \ref{cor:connectedness}) and so is itself a resolution. When $G \subset SL_2$, $\Hilb^r[\mb{A}^2/G]$ is a Nakajima quiver variety \cite[Theorem 2]{weiqiang} and so is connected \cite[Theorem 6.2]{quiverkacmoody}. The case for general $G$ is unknown to the authors.}.   When $r = 1$, $\Hilb^1[\mb{A}^2/G] \to \mb{A}^2/G$ is the minimal resolution \cite[Theorem 5.1]{kidoh} \cite[Theorem 3.1]{ishii}.

From now on, we restrict to $G$ abelian. In Section \ref{sec:CST} we will reduce our analysis to when the group is cyclic. To this end, we consider $G$ cyclic of order $n$ acting on $\mb{A}^2$ by $(x,y) \mapsto (\zeta^{a} x, \zeta^{b} y)$ where $\zeta$ is a primitive $n^{th}$ root of unity and $\gcd(a,b) = 1$. We will denote this group by $G_{a,b;n}$ and the equivariant Hilbert scheme $\Hilb^r[\mb{A}^2/G_{a,b;n}]$ by $H^r_{a,b;n}$. The first result of this article concerns the behavior of the compactly supported Betti numbers $b_i(H^r_{a,b;n})$. 

\begin{mainthma}\label{mainthma} Fix integers $r > 0$ and $a,b$ with $ab > 0$, or equivalently $a,b$ having the same sign. Then $b_i(H^r_{a,b;n}) = b_i(H^r_{a,b;n+ab})$ for all $n > rab$. 
\end{mainthma}

That is, the Betti numbers of $H^r_{a,b;n}$ are eventually periodic in $n$ with period $ab$. The proof of Main Theorem A uses the Bialynicki-Birula decomposition to stratify $H^r_{a,b;n}$ by locally closed affine cells. Thus the statement of Main Theorem A lifts to the Grothendieck ring of varieties $K_0(\mc{V}_\mb{C})$. 

\begin{thm}\label{thm:motivic} Fix integers $a,b$ with $ab > 0$. Then the class $[H^r_{a,b;n}]$ in $K_0(\mc{V}_\mb{C})$ is a polynomial in $\mb{L} = [\mb{A}^1]$ whose coefficients are periodic in $n$ with period $ab$ for $n > rab$. In particular, any motivic invariants of $H^r_{a,b;n}$ are eventually periodic in $n$.
\end{thm} 

\noindent When $a = b = 1$ and $n = 3$, this explains an observation of Gusein-Zade, Luengo, and Melle-Hernandez \cite[pg. 601]{GZLMH}. 

Our second main result examines the behavior of the topological invariants when $ab < 0$. Recall that a function $f: \mb{Z} \to \mb{Z}$ is called quasipolynomial of period $k$ if there are a polynomials $p_1, \ldots, p_k$ such that $f(n) = p_l(n)$ where $n \equiv l \mod k$. 

\begin{mainthmb}\label{mainthmb} Fix integers $r > 0$ and $a,b$ such that $ab < 0$, i.e. with opposite sign. Then the topological Euler characteristic $\chi_c(H^r_{a,b;n})$ is a quasipolynomial in $n$ with period $|ab|$ for all $n \gg 0$. 
\end{mainthmb}

\begin{rem} With a finer combinatorial analysis we can prove a strengthening of Main Theorem B to show that quasipolynomiality holds for Betti numbers and classes in the Grothendieck ring. Furthermore, one can show that quasipolynomiality holds for $n > r|ab|$ and that the quasipolynomial $\chi(H^r_{a,b;n})$ is of degree $r$. This will appear in forthcoming work. \end{rem}

\subsection{Background and motivation} 

Equivariant Hilbert schemes were first introduced by Ito and Nakamura \cite{in1} for finite subgroups $G \subset SL_2$. They play a central role in the Mckay correspondence (see for example \cite{reid, bkr, bezrufinkel}). Indeed much of the geometry of $\Hilb^r[\mb{A}^2/G]$ is determined in this case by the representation theory of $G$. On the other hand, very little is known about equivariant Hilbert schemes for general finite subgroups $G \subset GL_2$ (apart from the case $r = 1$, see for example \cite{kidoh, ishii}). 

This is the first paper in a project to understand the geometry of equivariant Hilbert schemes for abelian subgroups of $GL_2$ using the combinatorics of balanced partitions (Section \ref{sec:balanced}). The main theorems of this paper show new phenomena that appear only when we let the group vary outside of $SL_2$. These results are similar in spirit to the work of G\"ottsche \cite{gottsche}, Nakajima \cite{nakajima1} and others which show that one should study Hilbert schemes all at once, though in our case for all groups rather than for all $r$. 

Balanced partitions carry much more geometric information than just topological invariants. For example they determine an open affine cover of $\Hilb^r[\mb{A}^2/G_{a,b;n}]$  whose coordinate rings can be written purely combinatorially from the partitions (Section \ref{sec:geometry}). The hope is that the combinatorial bijections used in the proofs of Theorems \ref{thm:bijection} and \ref{thm:qpolynomial} have an interpretation on the level of the equivariant Hilbert schemes themselves that will lead to a geometric explanation for the periodicity and quasipolynomiality phenomena. 

\subsubsection{Toric resolutions and continued fractions}\label{sec:toric} The particular case of $r = 1$, $a = 1$ and $b = k > 0$ is instructive. Then $\mb{A}^2/G_{1,k;n}$ is the affine toric variety corresponding to the cone generated by $(0,1)$ and $(n,-k)$ and $\Hilb^1[\mb{A}^2/G_{1,k;n}]$ is the toric minimal resolution. It then follows from a result of Hirzebruch \cite[Theorem 10.2.3]{cls} that the Poincare polynomial of $H^r_{1,k;n}$ is of the form $P_{H^r_{1,k;n}}(z) = lz^2 + z^4$ where $l$ is the length of Hirzebruch-Jung continued fraction expansion 

$$
\frac{n}{k} = [[a_1, \ldots, a_l]] := a_1 - \frac{1}{a_2 - \frac{1}{a_3 \ddots \frac{1}{a_l}}}
$$

\noindent and $a_l > 1$.  This is evidently periodic in $n$ with period $k$. 

\begin{figure}[h]
\label{fig:supplementary}
\includegraphics[scale = .35]{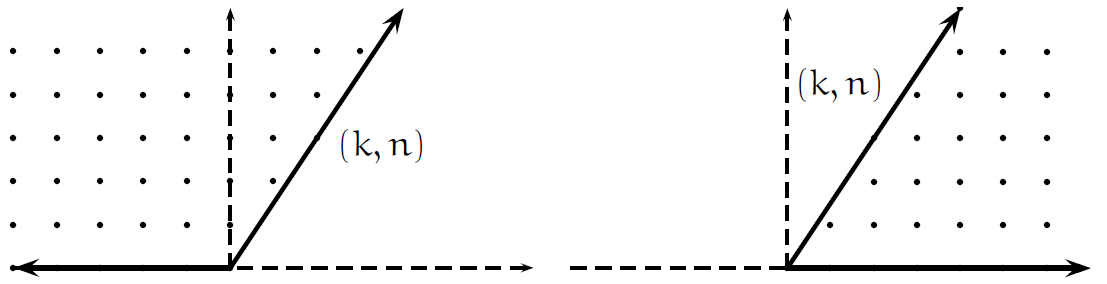}
\caption{\small The supplementary cones which correspond to the affine toric varieties $\mb{A}^2/G_{1,k;n}$ and $\mb{A}^2/G_{1,-k;n}$.}
\end{figure}

A similar computation when $r = 1$, $a = 1$ and $b = -k < 0$ yields the singular toric variety with supplementary cone. Then the Poincare polynomial takes the same form where $l$ is the length of the continued fraction expansion of $\frac{n}{n-k}$. Quasipolynomiality can then be deduced from a \textit{geometric} duality between the continued fractions of supplementary cones \cite[Proposition 2.7]{contfrac}. In fact, the length of the continued fraction expansion of $\frac{n}{n-k}$ is a linear quasipolynomial in $n$. 

For $r> 1$, we provide an analogue of the continued fraction expansion given by the set of balanced partitions defined below. We will see that the balanced partitions control the topology of the Hilbert scheme resolution of $\Sym^r(\mb{A}^2/G_{a,b;n})$ the same way the continued fraction controls the topology of the minimal resolution of $\mb{A}^2/G_{1,k;n}$. Furthermore, Theorems \ref{thm:bijection} and \ref{thm:qpolynomial} below, from which we deduce the main theorems, can be seen as a higher dimensional analogue of the geometric duality for continued fractions. 

\subsubsection{Future work and speculations} Ultimately, the goal is to understand the total cohomology 

$$
\mb{H}_{a,b;n} := \bigoplus_{r \geq 0} H_c^*(H^r_{a,b;n},\QQ)
$$

\noindent and compute its graded character which is the generating function of the Betti numbers $b_{i}(H^r_{a,b;n})$. When $(a,b) = (1,-1)$ so that $G_{1,-1;n} \subset SL_2$, $H^r_{1,-1;n}$ is diffeomorphic to $\Hilb^r(H^1_{1,-1;n})$ \cite[Lemma 4.1.3]{nagao} and the G\"ottsche formula \cite[Theorem 0.1]{gottsche} computes this generating function as an infinite product. After specializing to the Euler characteristic, we can deduce the formula

$$
\sum_{r \geq 0} \chi_c(H^r_{1,-1;n}) t^r = \left(\prod_{i \geq 1} \frac{1}{1 - t^i} \right)^n
$$ 

\noindent from the cores-and-quotients bijection (see Proposition \ref{prop:trivialcore}). 

The work of Nakajima \cite{nakajima1, quiverkacmoody} explains these infinite product formulas using representation theory of infinite dimensional Lie algebras. In particular, $\mb{H}_{1,-1;n}$ is a highest weight irreducible representation of a certain Heisenberg Lie algebra and this action intertwines two natural bases of $\mb{H}_{1,-1;n}$ coming from cores-and-quotients \cite{nagao}. We expect a similar picture to be true for the more general equivariant Hilbert schemes $H^r_{a,b;n}$. 

\begin{ques}\label{ques} Does $\mb{H}_{a,b;n}$ carry a natural action of an infinite dimensional Lie algebra $\mc{H}_{a,b;n}$ that can be described combinatorially in terms of balanced partitions?
\end{ques}

\noindent Computer computations with balanced partitions suggest the answer to Question \ref{ques} is yes and furthermore that $\mc{H}_{a,b;n}$ is generated in degrees $r$ for $rab < n$. This particular bound is interesting because it is the bound appearing in Main Theorem \hyperref[mainthma]{A}. This suggests that if Question \ref{ques} has an affirmative answer, then there is some relationship between the Lie algebras $\mc{H}_{a,b;n}$ and $\mc{H}_{a,b;n+ab}$ and their representations on the corresponding cohomologies at least when $ab > 0$. 

Moreover, these computations suggests that the Betti number generating function for $H^r_{a,b;n}$ is in general not an infinite product when $G_{a,b;n}$ is not in $SL_2$, but rather is a quasimodular form that can be written as a finite sum of infinite products. This is part of a general picture that generating functions for sheaf counting invariants on surfaces have modular properties (see \cite{modularforms} for a survey on this phenomena). Indeed the Euler characteristics and Poincar\'e polynomials of $H^r_{a,b;n}$ are naive Donaldson-Thomas type invariants \footnote{See for example \cite{bridgeland} and \cite{BBS} for Hilbert scheme invariants from the point of view of Donaldson-Thomas theory.} and the modularity property, if true, would be an analogue of \textit{S-duality} \cite{vafawitten} for the the quotient orbifolds $[\mb{A}^2/G_{a,b;n}]$. It would then be an interesting question to consider how the structure of these generating functions interacts with the stabilization properties from Main Theorems \hyperref[mainthma]{A} and \hyperref[mainthma]{B}.

\subsection{Balanced partitions}
\label{sec:balanced}

Main Theorems \hyperref[mainthma]{A} and \hyperref[mainthmb]{B} are proved by expressing the invariants above in terms of counting certain colored partitions or Young diagrams. We call these \textit{balanced} partitions. 

A partition $\lambda$ of a natural number $m$ is a sequence of nonnegative integers $\lambda_1 \geq \ldots \lambda_l \geq 0$ such that $\lambda_1 + \ldots + \lambda_l = m$. We identify $\lambda$ with its Young diagram, which is a subset of $m$ boxes arranged as left justified rows so that the $i^{th}$ row contains $\lambda_i$ boxes. We view this as living inside the $\mb{Z}_{\geq 0}^2$ lattice and use notation as in the diagram below. We denote by $l(k)$ (resp $c(h)$) the number of blocks in the $k^{th}$ row (resp $h^{th}$ column) of $\lambda$. 

\begin{figure}[h]
\includegraphics[scale=.4]{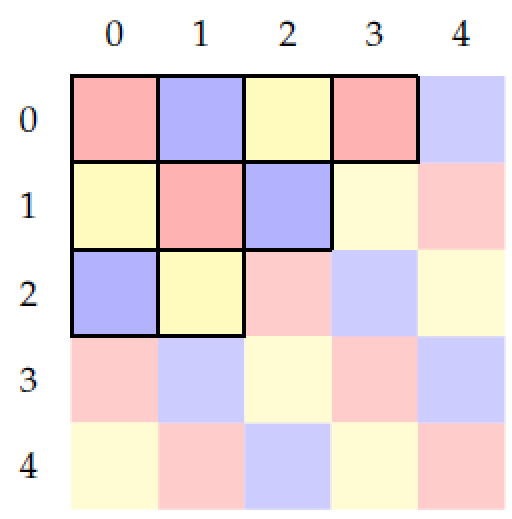}
\caption{\small The Young diagram corresponding to the partition $(4,3,2)$ of $9$. It is $(1,-1;3)$ balanced and the colors are the residue classes $\pmod{3}$.}
\end{figure}

Anticipating that the boxes $(i,j)$ correspond to monomials $x^iy^j$ having $G_{a,b;n}$-weight $ai + bj$, we color the partitions by the monoid homomorphism $w: \ZZ_{\geq 0}^2 \to \ZZ/n\ZZ$ that assigns $a i + b j \mod n$ to each $(i,j) \in \ZZ_{\geq 0}^2$. In particular this assigns a color viewed as an element of $G_{a, b; n}$ to each box in $\lambda$. We say $\lambda$ is an $(a,b;n)$\textit{-balanced} partition if there exists an $r$ such that $\lambda$ contains exactly $r$ boxes colored by $s$ for each residue class $s$ modulo $n$. In particular, any such $\lambda$ must be a partition of $rn$. 

Denote the set of all $(a,b;n)$-balanced partitions of $rn$ by $B^r_{a,b;n}$. There is a function $\beta : B^r_{a, b;n} \to \ZZ_{\geq 0}$ we call the \textit{Betti statistic} (Definition \ref{def:betti}). We will show the following proposition using the Bialynicki-Birula decomposition.

\begin{prop}\label{prop:betti} The Betti numbers of $H^r_{a,b;n}$ are given by

$$
b_i(H^r_{a,b;n}) = \#\{ \lambda \in H^r_{a,b;n} : 2\beta(\lambda) = i\}.
$$

\noindent In particular, the Poincar\'e polynomial $P_{H^r_{a,b;n}}(z)$ of $H^r_{a,b;n}$ satisfies 

$$
P_{H^r_{a,b;n}}(z) = \sum_{\lambda \in B^r_{a,b;n}} z^{2\beta(\lambda)}.
$$

\end{prop}

\noindent The main theorems will then follow from the following combinatorial results. 
 
\begin{thm}\label{thm:bijection} Fix integers $r > 0$ and $a,b$ with $ab > 0$. There is a natural bijection $B^r_{a,b;n} \to B^r_{a,b;n+ab}$ that preserves the Betti statistic for $n > rab$. \end{thm} 

\begin{thm}\label{thm:qpolynomial} Fix integers $r > 0$ and $a,b$ with $ab < 0$. The cardinality $\# B^r_{a,b;n}$ is a quasipolynomial in $n$ of period $|ab|$ for $n \gg 0$. \end{thm}

\subsection{Acknowledgments} The authors would like to thank T. Graber for suggesting this project and helping with the early stages. We are grateful to W. Hann-Caruthers for helping with the computational aspects that led us to conjecture the main theorems, and to L. Li for providing us with a draft of the unfinished manuscript \cite{lili} from which we learned many of the ideas in Sections 2 and 3. D.B. would like to thank his advisor D. Abramovich for his constant help and encouragement without which this paper would have never materialized. Finally, we would like to thank J. Ali, K. Ascher, S. Asgarli, D. Ranganathan and A. Takeda for many helpful comments on this draft. D.B. was partially supported by a Caltech Summer Undergaduate Research Fellowship and NSF grant DMS-1162367.

\section{The geometry of $H^r_{a,b;n}$}
\label{sec:geometry} 

In this section we give a systematic description of the geometry of $H^r_{a,b;n}$. We discuss the natural torus action on $H^r_{a,b;n}$ as well as smoothness and irreducibility. 

\subsection{Torus actions} 

The algebraic torus $T = (\mb{C}^*)^2$ acts naturally on $\mb{A}^2$ or equivalently on $\mb{C}[x,y]$ by $(t_1,t_2)(x,y) = (t_1x,t_2y)$. This induces an action on $\Hilb^m(\mb{A}^2)$ by pulling back ideals,

$$
(t_1,t_2)\cdot I = (\{f(t_1x, t_2y) : f \in I\}).
$$

\noindent The fixed points of this action are the doubly homogeneous ideals, that is, the monomial ideals. These are in one-to-one correspondence with partitions $\lambda$ of $m$ by the assignment

$$
\lambda \mapsto I_\lambda = (\{x^ry^s : (r,s) \in \ZZ_{\geq 0}^2 \setminus \lambda\}).
$$

\noindent Define $\mc{B}_\lambda = (\{x^hy^k : (h,k) \in \lambda\})$. It is clear that $\mc{B}_\lambda$ forms a basis for $\mb{C}[x,y]/I_\lambda$ so that $I_\lambda \in \Hilb^m(\mb{A}^2)$. 

Every monomial ideal is fixed by $G_{a,b;n}$. However, $I_\lambda \in H^r_{a,b;n}$ if and only if $\mb{C}[x,y]/I_\lambda = \mb{C}\mc{B}_\lambda$ is isomorphic as a $G_{a,b;n}$ representation to $\mb{C}[G_{a,b;n}]^r$. The space $\mb{C}\mc{B}_\lambda$ decomposes as a direct sum of irreducible representations $\mb{C}x^iy^j$ for $(i,j) \in \lambda$ each with weight $ai + bj \mod n$. Since $\mb{C}[G_{a,b;n}]$ decomposes as a direct sum of one copy of each irreducible representation, $\mb{C}[G_{a,b;n}]^r$ must have $r$ copies of each. Thus each weight must appear $r$ times in the decomposition of $\mb{C}\mc{B}_\lambda$ so we have proved the following:

\begin{lemma}\label{lemma:balanced} The $(\mb{C}^*)^2$-fixed points in $H^r_{a,b;n}$ are in one to one correspondence with $B^r_{a,b;n}$, the set of $(a,b;n)$-balanced partitions of $rn$. \end{lemma} 

\subsection{Local theory of Hilbert schemes} In this section we recall facts about the local geometry of $\Hilb^m(\mb{A}^2)$ following Haiman's description given in \cite{haiman}. 

One can define a torus invariant open affine neighborhood $U_\lambda$ of $I_\lambda$ given by

$$
U_\lambda := \{I \ : \ \mb{C}[x,y]/I \ \text{ is spanned by } B_\lambda\} \subset \Hilb^m(\mb{A}^2).
$$

\noindent The coordinate functions on $U_\lambda$ are given by $c^{l,s}_{i,j}(I)$ for $(i,j) \in \lambda$ and $(l,s) \in \ZZ_{\geq 0}^2$ where

\begin{equation}
\label{eq_1}
x^ly^s = \sum_{(i,j) \in \lambda} c^{l,s}_{i,j}(I) x^iy^j \mod I. 
\end{equation}

\noindent Multiplying $\hyperref[eq_1]{(1)}$ by $x$ we obtain 

$$
x^{l + 1}y^s = \sum_{(h,k) \in \lambda} c^{l,s}_{h,k} x^{h+1}y^k = \sum_{(h,k) \in \lambda} \sum_{(i,j) \in \lambda} c^{l,s}_{h,k} c^{h + 1,k}_{i,j} x^i y^j
$$

\noindent Therefore the coefficients satisfy the relations

\begin{equation}
\label{eq_2}
c^{l+1,s}_{i,j} = \sum_{(h,k) \in \lambda} c^{l,s}_{h,k} c^{h+1,k}_{i,j}.
\end{equation}

\noindent Similarly, we obtain the relation

\begin{equation}
\label{eq_3}
 c^{l,s+1}_{i,j} = \sum_{(h,k) \in \lambda} c^{l,s}_{h,k} c^{h,k+1}_{i,j} 
\end{equation}

\noindent by multiplying by $y$. 

We will often denote the function $c^{l,s}_{i,j}$ as an arrow on the on the $\mb{Z}_{\geq 0}^2$ grid pointing from box $(l,s) \in \mb{Z}_{\geq 0}^2$ to box $(i,j) \in \lambda$. These functions $c^{l,s}_{i,j}$ are torus eigenfunctions with action given by 

$$
(t_1,t_2)\cdot c^{l,s}_{i,j} = t_1^{l - i}t_2^{s-j} c^{l,s}_{i,j}.
$$

\noindent Consequently, $G_{a,b;n}$ acts by 

$$
c^{l,s}_{i,j} \mapsto \zeta^{a(l - i) + b(s-j)}c^{r,s}_{i,j}.
$$

\noindent The actions commute so that $\Hilb^{rn}(\mb{A}^2)^{G_{a,b;n}}$, and thus $H^r_{a,b;n}$, inherits a $(\mb{C}^*)^2$ action. 

For each box $(i,j) \in \lambda$, define two distinguished coordinate functions

\begin{equation}
\label{def1}
d_{i,j} := c^{l(j),j}_{i,c(i)-1} \enspace \enspace \enspace u_{i,j} := c^{i,c(i)}_{l(j)-1,j}
\end{equation}

\noindent where $l(j)$ is the size of the $j^{th}$ row and $c(i)$ the size of the $i^{th}$ column of $\lambda$. We can picture $d_{i,j}$ and $u_{i,j}$ as southwest and northeast pointing arrows hugging the diagram. Note that each diagram has $2m$ such distinguished arrows associated to it, two for each box.

\begin{figure}[h]
\includegraphics[scale=0.4]{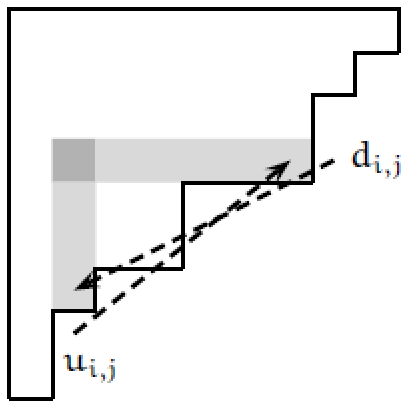}
\caption{\small The distinguished arrows $d_{i,j}$ and $u_{i,j}$ for the box $(i,j)$ in dark gray.}
\end{figure}

Now we can use these arrows to understand cotangent space to $I_\lambda \in \Hilb^m(\mb{A}^2)$ which we will denote $T_\lambda^*\Hilb^m(\mb{A}^2):= \mf{m}(I_{\lambda})/\mf{m}(I_{\lambda})^2$. The set of $c^{l,s}_{i,j}$ vanishing at $I_\lambda$ are precisely the ones for $(l,s) \notin \lambda$. These form generators for $T_\lambda^*\Hilb^m(\mb{A}^2)$. The relation $\hyperref[eq_2]{(2)}$ expresses $c^{l+1,s}_{i,j}$ as $c^{l,s}_{i-1,j} + \text{(higher order terms)}$ since $c^{i,j}_{i,j} \equiv 1$ and $c^{i',j'}_{i,j} \equiv 0$ for $(i,j) \neq (i',j') \in \lambda$. Thus

\begin{equation}
\label{eq_5}
c^{l+1,s}_{i,j} = c^{l,s}_{i-1,j} \mod \mf{m}(I_{\lambda})^2
\end{equation}

\noindent as local parameters in $T_\lambda^*\Hilb^m(\mb{A}^2)$. Similarly, $\hyperref[eq_3]{(3)}$ implies that 

\begin{equation}
\label{eq_6}
c^{l,s+1}_{i,j} = c^{l,s}_{i,j-1} \mod \mf{m}(I_{\lambda})^2
\end{equation}

\noindent in $T_\lambda^*\Hilb^m(\mb{A}^2)$. 

If we denote $c^{l,s}_{i,j}$ as an arrow, then $\hyperref[eq_5]{(5)}$ and $\hyperref[eq_6]{(6)}$ imply that if we slide an arrow horizontally or vertically while keeping $(l,s) \in \ZZ_{\geq 0}^2 \setminus \lambda$ and $(i,j) \notin \mb{Z}_{\geq 0}^2 \setminus \lambda$ then the arrow represents the same local parameter in $T^*_\lambda \Hilb^m(\mb{A}^2)$. Furthermore, if an arrow can be moved so that the head leaves the $\mb{Z}_{\geq 0}^2$ grid, then it is identically zero in $T^*_\lambda \Hilb^m(\mb{A}^2)$ because only positive degree monomials appear in $\CC[x,y]$. In this way every northwest pointing arrow vanishes in $T_\lambda^*\Hilb^m(\mb{A}^2)$ and any southwest or northeast pointing arrow can be moved until it either vanishes or is of the form $d_{i,j}$ or $u_{i,j}$ respectively. This proves the following:

\begin{prop}\label{prop:smooth}(\cite[Proposition 2.4]{haiman}, \cite[Theorem 2.4]{fogarty}) The set $\{d_{i,j},u_{i,j}\}$ over $(i,j) \in \lambda$ forms a system of local parameters generating the cotangent space of $I_\lambda \in \Hilb^m(\mb{A}^2)$. In particular, $\Hilb^m(\mb{A}^2)$ is smooth. \end{prop} 






\subsection{The cotangent space to $I_\lambda \in H^r_{a,b;n}$} 
\label{sec:cotangent}

We give a description of the weight space decomposition of the cotangent space to any monomial ideal $I_\lambda \in H^r_{a,b;n}$. This will be used later to compute the Bialynicki-Birula cells. 

By Proposition 2, $\Hilb^{rn}(\mb{A}^2)$ is smooth. It follows that the $G_{a,b;n}$-fixed locus is also smooth \cite[Proposition 4]{fogarty2}. In particular, the component $H^r_{a,b;n}$ is smooth. Moreover, since $G_{a,b;n}$ acts by scaling on $c^{l,s}_{i,j}$, then $c^{l,s}_{i,j}$ restricts to be nonzero on the fixed locus if and only if $G_{a,b;n}$ acts trivially on $c^{l,s}_{i,j}$. Thus the functions $c^{l,s}_{i,j}$ for $a(l - i) + b(s - j) \equiv 0 \pmod{n}$ generate the coordinate ring of $U_\lambda^{G_{a,b;n}}$. These correspond to the arrows that start and end on a box with the same color. We call these arrows \textit{invariant}. 

\begin{figure}[h]
\includegraphics[scale=0.4]{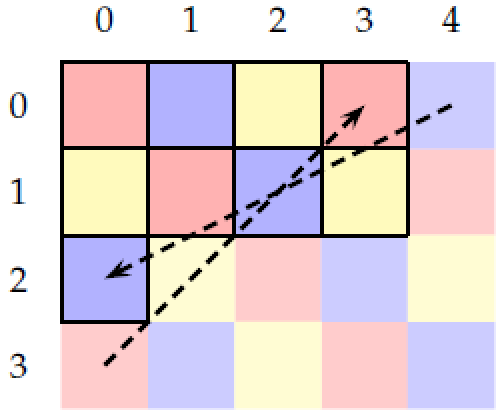}
\caption{\small Invariant arrows on $(4,3,2) \in B^3_{1,-1;3}$ corresponding to the box $(0,0)$.}
\end{figure}

\begin{prop}\label{prop:cotangent} Let $\lambda \in B^r_{a,b;n}$. The cotangent space to $I_\lambda \in H^r_{a,b;n}$ has basis given by the set of $d_{i,j}$ and $u_{i,j}$ that are invariant. \end{prop} 

\begin{proof} By the discussion above, these are the only local parameters of $\Hilb^{rn}(\mb{A}^2)$ that restrict to be nonzero in a neighborhood of $I_\lambda$ in $H^r_{a,b;n}$. On the other hand, $G_{a,b;n}$ acts trivially on the invariant arrows so they remain linearly independent in the cotangent space of the fixed locus.  \end{proof}

\begin{cor}\label{cor:2r} Let $\lambda$ be an $(a,b;n)$-balanced partition of $rn$. Then exactly $2r$ of the arrows of the form $d_{i,j}$ or $u_{i,j}$ are invariant. \end{cor} 

\begin{proof} The number of such arrows is the dimension of the cotangent space to $I_\lambda \in H^r_{a,b;n}$ which is $2r$ since $H^r_{a,b;n}$ is smooth and $2r$ dimensional. \end{proof}

Let $T_\lambda^* H^r_{a,b;n}$ denote the cotangent space to the torus fixed point $I_\lambda \in H^r_{a,b;n}$. Denote by $V(a,b)$ for $(a,b) \in \mb{Z}^2$ the irreducible representation of $(\mb{C}^*)^2$ on which $(t_1,t_2)$ acts by $t_1^at_2^b$. 

\begin{cor}\label{cor:weightspace} The weight space decomposition of $T_\lambda^* H^r_{a,b;n}$ as a representation of $(\mb{C}^*)^2$ is given by

$$
\bigoplus_{d_{i,j} \text{ invariant }} V(l(j) - i,j-c(i) + 1) \oplus \bigoplus_{u_{i,j} \text{ invariant }} V(i - l(j) + 1,c(i) - j).
$$


\end{cor}

\begin{proof} $(\mb{C}^*)^2$ acts on $c^{l,s}_{i,j}$ by $t_1^{l-i}t_2^{s - j}$. Then we get the result by Proposition $\ref{prop:cotangent}$ as well as the definition $\hyperref[def1]{(4)}$ of $d_{i,j}$ and $u_{i,j}$. \end{proof}

\begin{rem} In the literature, the weight space decomposition of the tangent space is often described in terms of the \textit{arm} and \textit{leg} of a box $(i,j) \in \lambda$. This description is equivalent because

$$
l(j) - i = \mathrm{arm}(i,j) + 1 \enspace \enspace \enspace c(i) - j = \mathrm{leg}(i,j) + 1.
$$

\end{rem}

\subsection{Connectedness of $H^r_{a,b;n}$}

We explain why $H^r_{a,b;n}$ is connected. The idea is that for any ideal $I \in H^r_{a,b;n}$, picking a monomial order $w$ and taking initial degeneration to a monomial ideal $I_0 := in_w I$ gives a rational curve in $H^r_{a,b;n}$ so that every ideal lies in the same connected component as a monomial ideal. Then one must show that all the monomial ideals are connected by chains of rational curves. This is done more generally in \cite{maclagan-smith} for multigraded Hilbert schemes. In this section we will deduce connectedness from the results of \cite{maclagan-smith}.

Let $R = \mb{C}[x,y] = \oplus_A R_a$ be the polynomial ring graded by some abelian group $A$. For any function $h : A \to \ZZ_{\geq 0}$, the \textit{multigraded Hilbert scheme} $\Hilb^h(R)$ is the subvariety of $\Hilb(\mb{A}^2)$ parametrizing homogeneous ideals $I \subset R$ such that

$$
\dim_\mb{C}(R/I)_a = h(a).
$$

\noindent That is, $\Hilb^h(R)$ is the moduli space of homogeneous ideals with Hilbert function $h$. 

The equivariant Hilbert scheme $H^r_{a,b;n}$ is a special case as follows. Let $G \subset GL_2$ be a finite abelian group and let $A$ be the dual group $\Hom(G,\mb{C}^*)$ of characters of $G$. Then the action of $G$ on $\mb{A}^2$ induces an $A$-grading on $R$ by

$$
R_a := \{p(x) \in R : g\cdot p(x) = a(g) p(x) \text{ for all } g \in G\}.
$$

\noindent It is easy to see that an ideal is homogeneous if and only if it is $G$-invariant. Furthermore, each $a \in A$ is the character of some irreducible representation of $G$ so the condition $R/I \cong \mb{C}[G]^r$ as representations of $G$ is equivalent to $\dim_{\mb{C}}(R/I)_a = r$ for each $a \in A$. Therefore

$$
\Hilb^h(R) = \Hilb^r[\mb{A}^2/G]
$$

\noindent where $R$ is $A$-graded by the action of $G$ and $h(a) = r$ for all $a$. 

Connectedness now follows from the following theorem of Maclagan and Smith:

\begin{thm}\label{thm:connectedness}(\cite[Theorem 3.15]{maclagan-smith}) $\Hilb^h(R)$ is rationally chain connected for any function $h : A \to \ZZ_{\geq 0}$ satisfying 

$$
\sum_{a \in A} h(a) < \infty.
$$

\end{thm}

\begin{cor}\label{cor:connectedness} $H^r_{a,b;n}$ is irreducible and the Hilbert-Chow morphism

$$
H^r_{a,b;n} \to \Sym^r(\mb{A}^2/G_{a,b;n})
$$

\noindent is a resolution of singularities. 

\end{cor}







\section{The Bialynicki-Birula stratification} 
\label{section3}

In this section we will show how to reduce the problem of computing Betti numbers of $H^r_{a,b;n}$ to counting $(a,b;n)$-balanced partitions of $rn$ with the Betti statistic (see Definition \ref{def:betti}). The idea is to use the action of an algebraic torus $(\mb{C}^*)^2$ on $H^r_{a,b;n}$ and the theory of Bialynicki-Birula \cite{bialynicki} to stratify $H^r_{a,b;n}$ into affine cells. Then a local analysis of the torus action at fixed points yields the appropriate statistic giving the Betti numbers. 

These techniques are standard in the theory of Hilbert schemes of points (see for example \cite{es1, es2, lili, buryak-feigin}). 

\subsection{The Bialynicki-Birula Decomposition Theorem} Let $S = (\mb{C}^*)^m$ be an algebraic torus and $X$ a smooth quasiprojective variety on which $S$ acts. Suppose the fixed point locus $X^S = \{p_1,\ldots, p_l\}$ is finite. Then for a generic one-dimensional subtorus $T \subset S$, we have $X^T = X^S$. We further assume that

$$
\lim_{t \to 0, t\in T} t\cdot x
$$

\noindent exists for all $x \in X$.

In this case, define

$$
X_j := \{ x : \lim_{t \to 0} t\cdot x = p_j\}.
$$

\noindent Then the $X_j$ are locally closed and $X = \bigsqcup_j X_j$. 

The action of $T$ on $X$ induces an action of $T$ on the tangent space $T_{p_j} X$. Define $T^+_{p_j} X$ to be the subspace of vectors on which $T$ acts with positive weight and let $n_j$ be its dimension. 

\begin{thm}\label{thm:bb} (Bialynicki-Birula Decomposition Theorem \cite[Theorem 4.4]{bialynicki}\footnote{Bialynicki-Birula originally proved this theorem for $X$ projective. The version we use here for quasiprojective $X$ is obtained by taking a torus equivariant compactification. See for example \cite[Lemma B.2]{BBS}}) Let $T \subset S$ and $X$ as above. Then each locally closed stratum $X_j$ is isomorphic to an affine space $\mb{A}^{n_j}$ so that 

$$
X = \bigsqcup_{j = 1}^l \mb{A}^{n_j}.
$$

\noindent Furthermore, the $i^{th}$ compactly supported Betti number $b_i = \dim H^i_c(X,\mb{Q})$ is given by

$$
\#\{j : 2n_j = i\}.
$$

\end{thm}

\subsection{The stratification of $H^r_{a,b;n}$} 

We will apply the above results to the action of $S = (\mb{C}^*)^2$ on $H^r_{a,b;n}$. As we saw (Lemma \ref{lemma:balanced}) the fixed points are indexed by balanced partitions $\lambda$. We pick $T = (t^{-p},t^{-q}) \subset S$ for generic $p \gg q > 0$ so that $(H^r_{a,b;n})^T$ consists of only the monomial ideals.  

\begin{lemma}\label{lemma:limits} For all $I \in H^r_{a,b;n}$, the limit 

$$
\lim_{t \to 0, t\in T} t\cdot I = I_0
$$

\noindent exists in $H^r_{a,b;n}$.

\end{lemma} 

\begin{proof} Consider the monomial partial order given by weight $(p,q)$. That is, $x^ly^s > x^iy^j$ if and only if $lp + sq > ip + jq$. Let $f \in I$ be any polynomial with leading term $x^ly^s$ under this monomial partial order. Then for $t \in T$,

$$
t\cdot f = t^{-(pl + qs)} x^ly^s + \sum_{pi + qj < pl + qs} t^{-(pi + qj)}x^iy^j \in t\cdot I.
$$

Multiplying through by $t^{pl + qs}$ gives $x^ly^s + t(\text{lower order terms}) \in t\cdot I$. So in the limit as $t \to 0$, we get $x^ly^s$. The dimension $\dim_\CC \CC[x,y]/I = rn$ is fixed so the degree of the polynomials in a Gr\"obner basis of $I$ is bounded \cite[Theorem 8.2]{grobner}. Since we are taking $p \gg q > 0$ all polynomials of bounded degree have a unique leading term under this monomial partial order so the limit ideal is the initial monomial ideal generated by these leading terms. Taking initial ideal is a flat limit so $I_0 \in H^r_{a,b;n}$ is a monomial ideal corresponding to some balanced partition. 

\end{proof}

Applying Theorem \ref{thm:bb} gives a decomposition of $H^r_{a,b;n}$ indexed by balanced partitions $\lambda \in B^r_{a,b;n}$:

$$
H^r_{a,b;n} = \bigsqcup_{\lambda \in B^r_{a,b;n}} \mb{A}^{n(\lambda)}
$$

\noindent where $n(\lambda)$ is the dimension of the positive weight subspace $T_\lambda^+H^r_{a,b;n} \subset T_\lambda H^r_{a,b;n}$ of the tangent space at $I_\lambda$. 

\begin{defn}\label{def:betti} Define the \textit{Betti statistic} function $\beta : B^r_{a,b;n} \to \ZZ_{\geq 0}$ as follows:

$$
\beta(\lambda) = \#\{d_{i,j} \textit{ invariant }\} + \#\{u_{i,j} \text{ invariant and vertical}\}.
$$

\noindent That is, $\beta(\lambda)$ is the number of invariant arrows on $\lambda$ that are pointing either strictly north or weakly southwest. 

\end{defn} 

\begin{rem} Note from the definition \hyperref[def1]{(4)} of $u_{i,j}$, it is vertical if and only if $i = l(j) - 1$. \end{rem} 

\begin{figure}[h]
\includegraphics[scale=0.4]{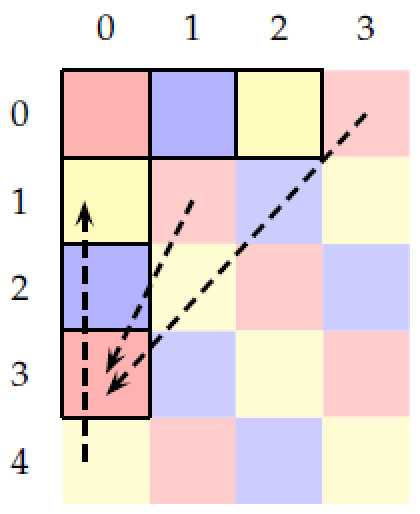}
\caption{\small This diagram has Betti statistic three.}
\end{figure}

\begin{prop}\label{prop:postiveweight} For any $\lambda \in B^r_{a,b;n}$, we have $\beta(\lambda) = \dim T_\lambda^+H^r_{a,b;n}$. \end{prop} 

\begin{proof} Corollary \ref{cor:weightspace} gives us the weight space decomposition of the cotangent space $T_\lambda^*H^r_{a,b;n}$. The tangent space $T_\lambda H^r_{a,b;n}$ is the dual space and so has weight space decomposition

$$
\bigoplus_{d_{i,j} \text{ invariant }} V(-(l(j) - i),-(j-c(i) + 1)) \oplus \bigoplus_{u_{i,j} \text{ invariant }} V(-(i - l(j) + 1),-(c(i) - j)).
$$

Considering the subtorus $T = (t^{-p},t^{-q})$, we see the weight spaces for this subtorus are generated by the invariant $d_{i,j}$ with weight $p(l(j) - i) + q(j - c(i) + 1)$ and invariant $u_{i,j}$ with weight $l(i - r(j) + 1) + q(c(i) - j)$. The $d_{i,j} = c^{r(j), j}_{i, c(i) - 1}$ arrows point southwest and so satisfy $l(j) > i$. Since $p \gg q > 0$, this means the weight $p(l(j) - i) + q(j - c(i) + 1) > 0$.

On the other hand, a $u_{i,j} = c^{i, c(i)}_{l(j) - 1, j}$ vector points northeast. If it points strictly northeast, then $c(i) > j$ but $l(j) - 1 < i$ and so the weight $p(i - l(j) + 1) + q(c(i) - j) < 0$. If it points strictly north, then $l(j) - 1 = i$ and $c(i) > j$ so that $p(i - l(j) + 1) + q(c(i) - j) = q(c(i) - j) > 0$. Therefore the positive weight vectors are exactly counted by the Betti statistic. 

\end{proof}

This proves Proposition \ref{prop:betti} which we repeat here for convenience:

\begin{prop1.1} The Betti numbers of $H^r_{a,b;n}$ are given by

$$
b_i(H^r_{a,b;n}) = \#\{ \lambda \in H^r_{a,b;n} : 2\beta(\lambda) = i\}.
$$

\noindent In particular, the Poincare polynomial $P_{H^r_{a,b;n}}(z)$ of $H^r_{a,b;n}$ satisfies 

$$
P_{H^r_{a,b;n}}(z) = \sum_{\lambda \in B^r_{a,b;n}} z^{2\beta(\lambda)}
$$

\noindent and the topological Euler characteristic is given by $\chi(H^r_{a,b;n}) = \# B^r_{a,b;n}$. 

\end{prop1.1}

This reduces Main Theorems \hyperref[mainthma]{A} and \hyperref[mainthmb]{B} to the combinatorial statements in Theorems \ref{thm:bijection} and \ref{thm:qpolynomial}. The proofs of these will be given in Section \ref{sec:proofs}.

\subsection{Grothendieck ring of varieties} In this section we will discuss the Grothendieck ring of varieties. Due to the Bialynicki-Birula decomposition, any statements about Betti numbers (for example Main Theorem \hyperref[mainthma]{A}) lift to the Grothendieck ring of varieties. 

Recall the \textit{Grothendieck ring of varieties} $K_0(\mc{V}_\mb{C})$ is the ring generated by isomorphism classes $[X]$ of varieties $X/\mb{C}$ under the \textit{cut-and-paste} relations: 

$$
[X] = [U] + [X \setminus U] \enspace \enspace \enspace  U \subset X \text{ open }. 
$$

\noindent The ring structure is given by $[X][Y] = [X \times Y]$ with unit $[pt] = 1$. We denote by $\mb{L} = [\mb{A}^1] \in K_0(\mc{V}_\mb{C})$. Then $[\mb{A}^n] = \mb{L}^n$.

If $X = \bigsqcup_i X_i$ where $X_i \subset X$ are a finite collection of locally closed subvarieties, then

$$
[X] = \sum_i [X_i].
$$

\noindent Thus the Bialynicki-Birula decomposition induces a decomposition of the class in $K_0(\mc{V}_\mb{C})$. We get the following: 

\begin{prop}\label{prop:motivic} The class of $H^r_{a,b;n}$ in $K_0(\mc{V}_\mb{C})$ is given by 

$$
[H^r_{a,b;n}] = \sum_{\lambda \in B^r_{a,b;n}} \mb{L}^{\beta(\lambda)}.
$$

\end{prop}

The ring $K_0(\mc{V}_\mb{C})$ is universal with respect to ring valued invariants of varieties satisfying cut-and-paste and splitting as a product for $X \times Y$. These include compactly supported Euler characteristic, virtual Poincare polynomials, and virtual mixed Hodge polynomials. These are often called motivic invariants.

Proposition \ref{prop:motivic} allows us to compute all motivic invariants of $H^r_{a,b;n}$ in terms of the Betti statistic on the set $B^r_{a,b;n}$ of balanced partitions. Then we apply Theorems \ref{thm:bijection} and \ref{thm:qpolynomial} proven below to obtain Theorem \ref{thm:motivic}. 

\section{Proofs of the theorems} 
\label{sec:proofs}

In Section \ref{section3}, we showed how the main theorems follow from Theorems \ref{thm:bijection} and \ref{thm:qpolynomial}. In this section we will give combinatorial proofs of these results after making an initial reduction. 

\subsection{The Chevalley-Shephard-Todd Theorem} 
\label{sec:CST}

Here we reduce to the case where both $a$ and $b$ are coprime to $n$ using the Chevalley-Shephard-Todd theorem. 

Let $G$ be a finite group acting linearly and faithfully on $\mb{A}^k$. We say that an element $\gamma \in G$ is a \textit{pseudoreflection} if it fixes a hyperplane in $\mb{A}^k$. We recall the following classical theorem:

\begin{thm}\label{thm:cst}(Chevalley-Shephard-Todd \cite[\S 5 Thm 4]{CST}) The following are equivalent:

\begin{enumerate}[(a)]

\item $G$ is generated by pseudoreflections, 

\item $\mb{A}^k/G$ is smooth, 

\item $\mb{A}^k/G \cong \mb{A}^k$,

\item the natural map $\mb{A}^k \to \mb{A}^k/G$ is flat. 

\end{enumerate} \end{thm}

Let $Y_{r,G} \subset \Hilb^r[\mb{A}^k/G]$ denote the irreducible component containing the locus of $r$ distinct free $G$-orbits in $\mb{A}^k$. 

\begin{cor}\label{cor:cst} The restriction $h_1: Y_{1,G} \to \mb{A}^k/G$ of the Hilbert-Chow morphism to $Y_{1,G}$ is an isomorphism if and only if any of the equivalent conditions of the Chevalley-Shephard-Todd theorem hold. \end{cor} 

\begin{proof} Suppose the conditions of the theorem hold so that $\mb{A}^k \to \mb{A}^k/G$ is flat. Then this is a flat family of $G$-orbits in $\mb{A}^k$ and so induces a map 

$$
\mb{A}^k/G \to \Hilb^1[\mb{A}^k/G]
$$

\noindent which is a section to $h$. This is an isomorphism on a dense open subset of $Y_{1,G}$ with inverse given by $h_1$ and so is an isomorphism everywhere. 

For the converse suppose $h_1 : Y_{1,G} \to \mb{A}^k/G$ is an isomorphism. We have a commutative diagram 

$$
\xymatrix{\mc{U}_1 \ar[r] \ar[d] & \mb{A}^k \ar[d] \\ Y_{1,G} \ar[r]^h & \mb{A}^k/G}
$$

\noindent where $\mc{U}_1$ is the universal family over $Y_{1,G}$ and $\mc{U}_1 \to \mb{A}^k$ is $G$-equivariant. The group $G$ acts fiberwise on $\mc{U}_1$ such that $\mc{U}_1/G = Y_{1,G}$ so the $G$-equivariant map $\mc{U}_1 \to \mb{A}^k$ over $\mb{A}^k/G$ induces an isomorphism $\mc{U}_1/G \cong \mb{A}^k/G$. It follows that $\mc{U}_1 \to \mb{A}^k$ is an isomorphism and $\mb{A}^k \to \mb{A}^k/G$ is flat so the equivalent conditions of the theorem hold. 

\end{proof} 

The above results allow us to reduce to the case where our group has no pseudoreflections. Let $H \subset G$ be the subgroup generated by pseodoreflections. First note that if $\gamma \in H$ fixes the hyperplane $\mc{H} \subset \mb{A}^k$, then $g\gamma g^{-1}$ fixes $g\mc{H}$ for any $g \in G$ so that $H$ is a normal subgroup. Denote $\widetilde{G} := G/H$. 

\begin{prop}\label{prop:cst} In the situation above, there is a natural morphism 

$$
\Hilb^r[\mb{A}^k/G] \to \Hilb^r[\mb{A}^k/\widetilde{G}]
$$

\noindent which induces an isomorphism $Y_{r,G} \cong Y_{r,\widetilde{G}}$. \end{prop} 

\begin{proof} We construct this isomorphism explicitly. Let $\mc{U}_r \to \Hilb^r[\mb{A}^k/G]$ be the universal family. It comes equipped with a $G$-equivariant map $\mc{U}_r \to \mb{A}^k$. The fiberwise quotient $\mc{U}_r/H \to \Hilb^r[\mb{A}^k/G]$ is a flat family of $\widetilde{G}$-equivariant schemes of length $r|G|/|H| = r |\widetilde{G}|$. To see flatness note that $\OO_{\mc{U}_r/H} = \OO_{\mc{U}_r}^H$ is a direct summand of the flat module $\OO_{\mc{U}_r}$ since we are in characteristic $0$.   

The natural map 

$$
\mc{U}_r/H \to \mb{A}^k/H \cong \mb{A}^k
$$

\noindent induces a map $\mc{U}_r/H \to \Hilb^r[\mb{A}^k/G] \times \mb{A}^k$. We need to check that this is an embedding or equivalently that the $G$-equivariant morphism of $\OO_{\Hilb^r[\mb{A}^k/G]}$-algebras 

$$
\CC[\mb{A}^n]^H \otimes \OO_{\Hilb^r[\mb{A}^k/G]} \to \OO_{\mc{U}_r}^H
$$

\noindent is surjective.  We can check surjectivity on fibers; over the point $[J] \in \Hilb^r[\mb{A}^k/G]$ corresponding to some $G$-invariant ideal, this map is just $\CC[\mb{A}^n]^H \to (\CC[\mb{A}^n]/J)^H = \CC[\mb{A}^n]^H/(J \cap \CC[\mb{A}^n]^H)$ which is surjective.  

Moreover the regular representation is preserved by taking invariants, $\mb{C}[G]^H \cong_{\widetilde{G}} \mb{C}[\widetilde{G}]$. Thus $\mc{U}_r/H \to \Hilb^r[\mb{A}^k/G]$ is a flat family of $\widetilde{G}$-equivariant subschemes of $\mb{A}^k$ of length $r|\widetilde{G}|$ carrying $r$ copies of the regular representation and so induces a morphism 

$$
\varphi : \Hilb^r[\mb{A}^k/G] \to \Hilb^r[\mb{A}^k/\widetilde{G}]. 
$$

To construct an inverse over $Y_{r,\widetilde{G}}$, take the universal family

$$
\xymatrix{\mc{V}_r \ar[r]^\rho \ar[d] & \mb{A}^k \\ Y_{r,\widetilde{G}} & }
$$

\noindent Since $H$ is generated by pseudoreflections, pulling back the quotient map $\mb{A}^k \to \mb{A}^k/H \cong \mb{A}^k$ along $\rho$ gives a flat family $\mc{Z}_r$:

$$
\xymatrix{\mc{Z}_r = \mc{V}_r \times_{\mb{A}^k} \mb{A}^k \ar[r] \ar[d] & \mb{A}^k \ar[d]^{\mod H} \\
\mc{V}_r \ar[r]^\rho \ar[d] & \mb{A}^k \\
Y_{r,\widetilde{G}} & }
$$

\noindent A general fiber of $\mc{Z}_r \to Y_{r,\widetilde{G}}$ consists of $r$ distinct free $G$ orbits so by flatness every fiber carries $r$ copies of the regular representation of $G$. Furthermore closed embeddings are stable under base change so $\mc{Z}_r \to Y_{r,\widetilde{G}}$ is a flat family of subschemes of $\mb{A}^k$ inducing a morphism 

$$
\psi : Y_{r,\widetilde{G}} \to Y_{r,G} \subset \Hilb^r[\mb{A}^k/G].
$$ 

\noindent Since $\psi$ and $\varphi$ are inverses on the dense open subset parametrizing $r$ distinct free orbits they give an isomorphism everywhere. 

\end{proof}

\begin{rem} Note that in the above proof, there is always a morphism $\Hilb^r[\mb{A}^k/G] \to \Hilb^r[\mb{A}^k/\widetilde{G}]$ for any normal subgroup $H \subset G$. The fact that $H$ is generated by pseudoreflections is only used in constructing the inverse over $Y_{r,\widetilde{G}}$. 
\end{rem}

Proposition \ref{prop:cst} justifies our restriction to considering only the cyclic subgroups $G_{a,b;n} \subset GL_2$. Indeed if $G \subset GL_2$ is any abelian subgroup with no pseudoreflections then it must be cyclic \cite[Satz 2.9]{brieskorn}. By Corollary \ref{cor:connectedness} $\Hilb^r[\mb{A}^2/G] = Y_{r,G}$ for $G$ abelian. Consequently, every equivariant Hilbert scheme for an abelian group action on $\mb{A}^2$ is isomorphic to $\Hilb^r[\mb{A}^2/G_{a,b;n}]$ for some $a,b$ and $n$. 

\begin{cor}\label{cor:coprime} Suppose one of $a$ and $b$, say $a$ without loss of generality, is not coprime to $n$ so that $a = da'$ and $n = dn'$. Then $H^r_{a,b;n} \cong H^r_{a',b;n'}$. \end{cor} 

\begin{proof} The generator of $G_{a,b;n}$ satisfies

$$
\left( \begin{array}{cc} \zeta_n^a & 0 \\ 0 & \zeta_n^b \end{array} \right)^{n'} = \left(\begin{array}{cc} 1 & 0 \\ 0 & \zeta_n^{bn'}\end{array} \right).
$$

\noindent This is a nontrivial pseudoreflection generating a cyclic subgroup $H \subset G_{a,b;n}$ of order $d$. The quotient $G_{a,b;n}/H$ is a cyclic group of order $n'$ acting by weights $a'$ and $b$ so we get the required isomorphism.  
\end{proof} 

In light of Corollary \ref{cor:coprime}, it suffices to consider only the case when $n$ is coprime to $a$ and $b$. Indeed if $n = dn'$ and $a = da'$, then sending $n$ to $n + |ab|$ is equivalent by the corollary to sending $n'$ to $n' + |a'b|$ and so is compatible with periodicity and quasipolynomiality statements with period $|ab|$. We will only need this reduction in the proof of Theorem \ref{thm:qpolynomial}. 

\subsection{Proof of Theorem \ref{thm:bijection}}

Recall we are going to prove the following: 

\begin{thm1.2} Fix integers $r > 0$ and $a,b$ with $ab > 0$. There is a natural bijection $B^r_{a,b;n} \to B^r_{a,b;n+ab}$ that preserves the Betti statistic for $n > rab$.
\end{thm1.2} 

Since the involution $(a,b) \mapsto (-a,-b)$ does not change the family of cyclic groups $\{G_{a,b;n}\}_n$, we assume without loss of generality that $a,b > 0$. For any $(a,b;n)$-balanced partition $\lambda$ and $0 \le k \le n-1$, we denote by

$$
S_k := \{(i,j) \in \lambda \ | \ ai + bj = k \mod n\}
$$ 

\noindent the set of boxes in $\lambda$ labeled by $k \pmod{n}$. We also denote by
$$
D_k := \{(i,j) \in \ZZ_{\geq 0}^2 \ | \ ai + bj = k\}
$$

\noindent the \textit{$k^{th}$ diagonal}.

First we need the following lemmas:

\begin{lemma}\label{lemma:armandleg}
Suppose $\lambda$ is an $(a,b;n)$-balanced partition and $n> rab$. Let $k$ be an integer with $rab \le k \le n - 1$ satisfying $k = rab + au + bv$ for some nonnegative integers $u$ and $v$. Then the set $S_k$ can be split into two disjoint sets $A_k$ and $B_k$ which satisfy the properties:

\begin{enumerate}
\item If $(i,j)\in A_k$ then either $i < b$ or $(i-b,j+a)\in A_k$.
\item If $(i,j) \in B_k$ then either $j < a$ or $(i+b,j-a)\in B_k$.
\end{enumerate}

\end{lemma}
\begin{proof}
First notice that for $rab \le k \le n -1$ satisfying either $(i)$ or $(ii)$, the number of $(i,j)\in \ZZ_{\geq 0}^2$ such that $ai+bj=k$ is at least $r+1$, therefore there exist $(i_0(k),j_0(k))$ such that $ai_0(k)+bj_0(k)=k$ and $(i_0(k),j_0(k))\notin \lambda$. This means that the entries labeled $k \pmod{n}$ split into two sets, $A_k$ and $B_k$ defined by 

\begin{align*}
A_k &:= \{(i,j) \in \lambda \ | \ ai + bj = k \mod n, \ i \le i_0(k)\} \\
B_k &:= \{(i,j) \in \lambda \ | \ ai + bj = k \mod n, \ j \le j_0(k)\}
\end{align*}

\noindent Note that these sets must be disjoint because if $(i,j) \in \lambda$ satisfies $i < i_0(k)$ and $j < j_0(k)$ then $ai + bj < k < n$ so $(i,j) \notin S_k$.

We define a map $\varphi_k : S_k \to S_{k - ab}$ by $(i,j)\in A_k$ maps to the entry $(i,j-a)$ and $(i,j)\in B_k$ maps to the entry $(i-b,j)$. We first claim this gives an injective map. The map $\varphi_k$ is clearly injective on each set $A_k$ or $B_k$ individually so suppose there is $(i,j) \in A_k$ and $(i',j') \in B_k$ such that $\varphi_k(i,j) = \varphi_k(i',j')$. Then $(i, j-a) = (i' - b, j')$ so $(i,j)$ and $(i',j')$ are on the corners of an $b \times a$ rectangle. The label $k \pmod{n}$ appears in such a rectangle at most twice, namely at $(i,j)$ and $(i',j')$. This contradicts the fact that there is a box $(i_0, j_0) \notin \lambda$ labeled by $k$ with $i \le i_0$ and $j \le j_0$. 

Since $\lambda$ is a balanced partition, the number of boxes with each label have the same cardinality and so the injective map $\varphi_k$ must in fact be bijective. Suppose for the sake of contradiction that the first condition in the lemma is violated, so we have $(i,j)\in A_k$ but $i \geq b$ and $(i-b,j+a)\notin A_k$. Then the entry $(i-b,j)$ would be in $\lambda$, labeled $k-ab$ but $\varphi^{-1}_k(i-b,j) \notin \lambda$, which gives a contradiction. We can argue similarly for the entries in $B_k$.
\end{proof}

\begin{figure}[h]
\includegraphics[scale=0.4]{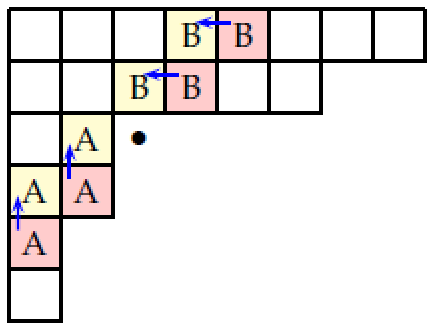}
\caption{\small This figure illustrates the map $\varphi_4$ on a $(1,1;5)$-balanced partition. The box marked by a bullet is the $(i_0,j_0) \notin \lambda$ with $ai_0 + bj_0 = 4$ that we use to define the sets $A_4$ and $B_4$.}
\end{figure}

\begin{lemma}\label{lemma:unique} The decomposition $S_k = A_k \cup B_k$ above does not depend on a choice of $(i_0(k),j_0(k)) \notin \lambda$ on the $k^{th}$ diagonal. In particular the decomposition is the unique one for which the map $\varphi_k$ is a bijection. \end{lemma}

\begin{proof} Suppose for contradiction that there was an $(i_1(k),j_1(k)) \in D_k$ with $(i_1(k),j_1(k))$ and a decomposition of $S_k$ as

\begin{align*}
A_k' &:= \{(i,j) \in \lambda \ | \ ai + bj = k \mod n, \ i \le i_1(k)\} \\
B_k' &:= \{(i,j) \in \lambda \ | \ ai + bj = k \mod n, \ j \le j_1(k)\}
\end{align*}

\noindent that is different than the one induced by $(i_0(k),j_0(k))$. Then there must be some $(i',j'),(i'',j'') \in \lambda \cap D_k$ such that $i_0 \le i' - b < i' \le i'' < i'' + b \le i_1$ with $(i' - b, j' + a), (i'' + b, j' - a) \notin \lambda$. That is $(i',j'), (i'',j'') \in B_k$ but $(i',j'), (i'', j'') \in A_k'$. In this case, we can compute 
\begin{align*}
\varphi_k(i',j') &= (i' - b, j') \in S_{k - ab}  \\ 
\varphi_k(i'',j'') &= (i'' - b, j'') \in S_{k - ab}
\end{align*} 

\noindent On the other hand, $(i'',j'' - a) \in S_{k - ab}$ and $\varphi_k^{-1}(i'',j''-a) = (i'' + b, j'' - a) \notin \lambda$ contradicting that $\varphi_k$ is a bijection. 

\end{proof}

\begin{cor}\label{cor:armandleg} Let $(i_0,j_0) \in D_{rab}$ such that $(i_0,j_0) \notin \lambda$. Define 
\begin{align*}
A = \{(i,j) \ | \ i < i_0, j \geq j_0\} \\
B = \{(i,j) \ | \ i \geq i_0, j < j_0\}
\end{align*}

\noindent Then for any $k$ satisfying the conditions of Lemma \ref{lemma:armandleg}, $A_k = A \cap S_k$ and $B_k = B \cap S_k$ and the maps $\varphi_k$ extend to an injective map $\varphi : A \cup B \to \ZZ_{\geq 0}^2$ that is surjective onto $S_{k - ab}$ for all such $k$. 

\end{cor} 

\begin{proof} First notice that $A$ and $B$ are disjoint. Since $rab \le k < n$ it follows that $S_k \subset A \cup B$. 

The boxes $(i,j)$ with $i \geq i_0$ and $j \geq j_0$ are not in $\lambda$. On the other hand, for every such $k$ there exists $(i',j') \in D_k$ with $i' \geq i_0$ and $j' \geq j_0$. Consequently for $(i,j) \in \lambda$, if $i < i'$ then $i < i_0$ and similarly if $j < j'$ then $j < j_0$ so that $A_k \subset A$ and $B_k \subset B$. The first result follows. Finally, we can extend the maps $\varphi_k$ by defining 

$$
\varphi(i,j) = \left\{ \begin{array}{lr} (i,j-a) & \text{if } (i,j) \in A \\ (i - b,j) & \text{if } (i,j) \in B\end{array} \right.
$$ 

\end{proof}

\begin{proof}[Proof of Theorem \ref{thm:bijection}]
We construct a bijection $B^r_{a,b;n} \to B^r_{a,b;n + ab}$. Let $\lambda \in B^r_{a,b;n}$. We will add boxes to the columns and rows of $\lambda$ as follows: 

\begin{enumerate}
\item If $(i,j)\in A_k$ and $n - b \le k \le n - 1$, then increase the length of column $j$ by $a$ boxes;
\item If $(i,j)\in B_k$ and $n - a \le k \le n - 1$, then increase the length of row $i$ by $b$ boxes.
\end{enumerate}

We can check that the process terminates in a partition since Lemma \ref{lemma:armandleg} guarantees that if $i>i'$, column $i'$ had at least as many boxes added as column $i$, and similarly for rows. We call the resulting partition $\psi(\lambda)$. Note that $\lambda \subset \psi(\lambda)$ as a subset of $\ZZ_{\geq 0}^2$.

We need to check that $\psi(\lambda) \in B^r_{a,b;n+ab}$. We can interpret the algorithm above as inserting $a$ boxes directly below each $(i,j) \in A_k$ with $n - b \le k \le n - 1$ and inserting $b$ boxes directly to the right of each $(i,j)\in B_k$ with $n - a \le k \le n - 1$. It is clear that these new boxes are labeled with $(a,b;n + ab)$-weight in the range $n \le k \le n + ab - 1$ and that the boxes of $\lambda$ are in bijection with the boxes of $\psi(\lambda)$ labeled with $(a,b; n+ab)$-weight in the range $0 \le k \le n -1$. Thus it suffices to check that we have inserted $r$ boxes of each weight $n \le k \le n + ab - 1$. Fix $k$ with $n \le k \le n + ab - 1$ and let 

$$
R_k = \{(i,j) \in \psi(\lambda) \ | \ ai + bj = k \pmod{n + ab} \}.
$$ 

\noindent Then $R_k \subset A \cup B$ and the restriction $\varphi : R_k \to S_{k - ab}$ is a bijection. Consequently, $\# R_k = \# S_{k - ab} = r$ since $\lambda$ is balanced and $\psi(\lambda) \in B^r_{a,b;n + ab}$. 

To check that $\psi$ is a bijection we start with $\mu \in B^r_{a,b;n + ab}$ and produce a $\lambda \in B^r_{a,b;n}$ with $\psi(\lambda) = \mu$. Indeed we obtain $\lambda$ by deleting all boxes of $\mu$ labeled by $k \pmod{n + ab}$ with $n \le k \le n + ab - 1$. Then the boxes of $\lambda$ labeled with $(a,b;n)$-weight $k$ correspond to the boxes of $\mu$ labeled with $(a,b;n + ab)$-weight $k$ with $0 \le k \le n - 1$ and it is clear that we can recover $\mu$ by inserting back in the boxes with larger $(a,b;n + ab)$-weight, i.e., $\psi(\lambda) = \mu$. 

Lastly, we need to check that the Betti statistic is preserved. First note that $\psi$ sends invariant arrows to invariant arrows without changing the direction (though possibly changing the slope) of the arrow. This is because $\psi$ induces a bijection on the boxes with labels $0 \le k \le n-1$ so stretching the arrow by applying $\psi$ doesn't affect whether it is invariant. On the other hand we can shrink any invariant arrow of $\psi(\lambda)$ onto an invariant arrow of $\lambda$ by moving the head and tail by the number of boxes deleted from $\psi(\lambda)$ to obtain $\lambda$, i.e., by applying $\varphi$ to the head and tail of the arrow. 

\end{proof}

For clarity, we will illustrate the above proof in the following example where $(a,b) = (2,3)$, $r = 2$ and $n = 13$. Consider the $(2,3;13)$-balanced partition below:

\begin{figure}[h]
\includegraphics[scale=0.4]{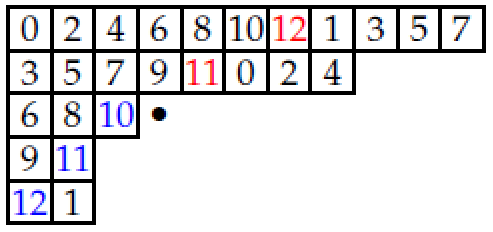}
\end{figure}

\noindent We have labeled the boxes by the weight $k \pmod{n}$. The box labeled by a $\bullet$ is the box in the diagonal $D_{12}$ that is not contained in $\lambda$ which we use to decompose $\lambda$ into the sets $A$ and $B$. The boxes $(i,j) \in A$ labeled with $n - b \le k \le n - 1$ are colored blue and the boxes $(i,j) \in B$ labeled with $n - a \le k \le n -1$ are colored green. 

Applying $\psi$ gives us the following $(2,3;19)$-balanced partition. 

\begin{figure}[h]
\includegraphics[scale=0.4]{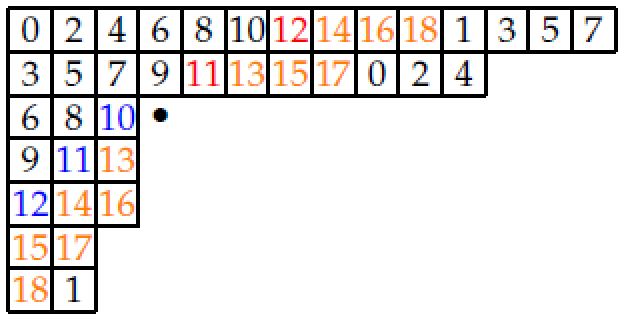}
\end{figure}

\noindent The new boxes that are inserted by $\psi$ are colored in orange and these are exactly the boxes with labels $n \le k \le n + ab - 1$. 

\subsection{Proof of Theorem \ref{thm:qpolynomial}} 

Recall we are going to prove the following: 

\begin{thm1.3} Fix integers $r > 0$ and $a,b$ with $ab < 0$. The cardinality $\# B^r_{a,b;n}$ is a quasipolynomial in $n$ of period $|ab|$ for $n \gg 0$.
\end{thm1.3}

The idea of the proof is to translate from $(a,b;n)$-balanced partitions to $(1,-1;n)$-balanced which we can then count using the cores-and-quotients bijection (see for example \cite[Chapter 11]{loehr}).

By Corollary \ref{cor:coprime} it suffices to prove Theorem \ref{thm:qpolynomial} for $a$ and $b$ both coprime to $n$. Without loss of generality, we assume that $a$ is positive and $b$ is negative. Let $P_m$ be the set of partitions of $m$. Consider the map 

$$
f : P_{rn} \to P_{-rabn}
$$

\noindent where for each $\lambda \in P_{rn}$, $f(\lambda)$ is the partition of $-rabn$ obtained by replacing each box of $\lambda$ by an $a \times (-b)$ rectangle. 

If $\lambda$ is the partition with rows $(\lambda_1, \ldots, \lambda_l)$, then $f(\lambda)$ is the partition with rows

\begin{equation}
(\underbrace{a\lambda_1, \ldots, a\lambda_1}_{\text{$-b$ times}}, \ldots, \underbrace{a\lambda_l, \ldots, a\lambda_l}_{\text{$-b$ times}}) \tag{$\ast$}.
\end{equation}

\noindent That is, $f(\lambda)$ is a partition whose row lengths are multiples of $a$ and such that each row repeats a multiple of $-b$ times. 

\begin{prop}\label{prop:expand} If $a$ and $b$ are both prime to $n$, then the function $f$ restricts to a bijection between $(a,b;n)$-balanced partitions of $rn$ and $(1,-1;n)$-balanced partitions of $-rabn$ satisfying condition $(\ast)$. \end{prop}

\begin{proof} It is clear that the map $f$ is injective and that it surjects onto the set of partitions satisfying $(\ast)$ so that $f$ is a bijection between $P_{rn}$ and the subset of $P_{-rabn}$ satisfying $(\ast)$. To prove the claim, it suffices to show that $\lambda$ is $(a,b;n)$-balanced if and only if $f(\lambda)$ is $(1,-1;n)$-balanced. 

We will use the generating function for the $G_{a,b;n}$-weights of the boxes of a partition. For any partition $\mu$, define 

$$
w^{\mu}_{a,b;n}(q) := \sum_{(i,j) \in \mu} q^{ai + bj \mod n} \in \QQ[q]/(q^n - 1).
$$

\noindent Now $\mu$ is $(a,b;n)$-balanced if and only if 

$$
w^\mu_{a,b;n}(q) = r(q^{n-1} + q^{n-2} + \ldots + 1) \pmod{q^n - 1}
$$ 

\noindent if and only if $(q - 1)w^\mu_{a,b;n}(q) = 0 \pmod{q^n - 1}$. 

Let $w^\lambda_{a,b;n}(q) = \sum_{l = 1}^{rn} q^{w_k} \pmod{q^n - 1}$ where $l$ runs through the $rn$ boxes of $\lambda$. Under the map $f$, each box gets replaced with an $a \times (-b)$ rectangle. Such a rectangle has $(1,-1;n)$-weight generating function given by 

$$
\left( \frac{q^a - 1}{q - 1} \right)\left(\frac{q^{-b} - 1}{q - 1}\right) \mod (q^n - 1)
$$
\noindent However, if the $k^{th}$ box of $\lambda$ has $(a,b;n)$-weight $w_k$, the corresponding rectangle in $f(\lambda)$ has a box with $(1,-1;n)$-weight $w_k - b + 1$ in the bottom right corner so the above generating function must be shifted by $q^{-b + 1}$. Putting this together gives  
$$
w^{f(\lambda)}_{1,-1;n}(q) = q^{-b + 1} w^\lambda_{a,b;n}(q) \left( \frac{q^a - 1}{q - 1} \right)\left(\frac{q^{-b} - 1}{q - 1}\right) \mod (q^n - 1). 
$$

\noindent for the $(1,-1;n)$-weight generating function of $f(\lambda)$. 

By assumption $a$ and $b$ are coprime to $n$ so $\frac{q^a - 1}{q - 1}$ and $\frac{q^{-b} - 1}{q - 1}$ are units in $\QQ[q]/(q^n - 1)$. It follows that $(q - 1)w^{f(\lambda)}_{1,-1;n}(q) = 0 \pmod{q^n - 1}$ if and only if $(q - 1)w^\lambda_{a,b;n}(q) = 0 \pmod{q^n - 1}$. Therefore $\lambda$ is $(a,b;n)$-balanced if and only if $f(\lambda)$ is $(1,-1;n)$-balanced.

\end{proof} 

\begin{prop}\label{prop:qpolynomial} Suppose $a,b$ are coprime to $n$ with $a>0$ and $b<0$. For any $r$, the number of $(a,b;n)$-balanced partitions of $rn$, $\#B^r_{a,b;n}$, is a quasipolynomial in $n$ with period $|ab|$.
\end{prop} 

\begin{proof} 

By Proposition \ref{prop:expand}, it suffices to count $(1,-1;n)$-balanced partitions of $-rabn$ satisfying $(\ast)$. 

Let $\lambda \in B^{-rab}_{1,-1;n}$ be a $(1,-1;n)$-balanced partition. By Proposition \ref{prop:trivialcore} in Appendix \ref{appendix}, such partitions are in bijection with $n$-tuples of partitions $Q_n(\lambda) = (\lambda^{(0)}, \ldots, \lambda^{(n-1)})$ such that $\sum |\lambda^{(i)}| = -rab$. Let $\mathrm{Ab}(\lambda)$ and $\mathrm{Ab}_n(\lambda)$ be the corresponding abacus and abacus with $n$ runners (see Appendix \ref{appendix}). Condition $(\ast)$ means that every subsequence of consecutive $1$'s in $\mathrm{Ab}(\lambda)$ has length divisible by $-b$ and every subsequence of consecutive $0$'s has length divisible by $a$. Increasing $n$ corresponds to increasing the number of rows of $\mathrm{Ab}_n(\lambda)$, or equivalently the number of partitions in the $n$-tuple $Q_n(\lambda)$. 

We write the $n$-quotient $Q_n(\lambda)$ of $\lambda \in B^{-rab}_{1,-1;n}$ as 
$$
(\mu_1,d_1,\mu_2,d_2,\dots,d_{s-1},\mu_s)
$$ 

\noindent where $\mu_i$ is a sequence of nonempty partitions and $d_i \in \ZZ_{\geq 0}$ stands for a sequence of $d_i$ empty partitions. The $d_i$ correspond to $d_i$ consecutive rows of $\mathrm{Ab}_n(\lambda)$ of the form $\ldots 111|000 \ldots$ where the position marked by the $|$ is the \textit{center} of the row. Each $\mu_i$ corresponds to consecutive rows that are not of this form and we call the $\mu_i$ \textit{chunks}.  

The congruence conditions on the subsequences of consecutive $1$'s and $0$'s on the single abacus $\mathrm{Ab}(\lambda)$ is equivalent to the same congruence condition on the length of consecutive $1$'s and $0$'s in $\mathrm{Ab}_n(\lambda)$ read in lexicographic order down each column. The condition that the core $C_n(\lambda)$ is empty means that after transposing every occurrence of $01$ in $\mathrm{Ab}_n(\lambda)$, the resulting abacus $\widetilde{\mathrm{Ab}}_n(\lambda)$ is of the form 

\begin{align}
\begin{split}
\label{alignment}
&\vdots \\
\ldots 111&1|000 \ldots \\ 
\ldots 111&1|000 \ldots \\
\ldots 111&|0000 \ldots \\
\ldots 111&|0000 \ldots \\ 
&\vdots 
\end{split}
\end{align} 

\noindent where the centers of each row are aligned except in at most one position.

For each collection of chunks $\{\mu_1, \ldots, \mu_s\}$, suppose that there is some $n$ so that these chunks can be used to construct an $n$-tuple $(\mu_1,d_1,\mu_2,d_2,\dots,d_{s-1},\mu_s)$ representing a balanced partition $\lambda$ satisfying $(\ast)$. Then we can add $|ab|$ to any one of the $d_i$ to obtain an $(n + |ab|)$-tuple of partitions representing a balanced partition in $B^{-rab}_{1,-1; n + |ab|}$ satisfying condition $(\ast)$. Indeed this corresponds to adding $|ab|$ consecutive rows of the form $\ldots 111|000 \ldots$ to the abacus $\mathrm{Ab}_n(\lambda)$ to obtain an abacus $\mathrm{Ab}_{n + |ab|}(\lambda')$ for some $(1,-1;n + |ab|)$-balanced partition $\lambda'$. The condition $\hyperref[alignment]{(7)}$ fixes how the centers of these new rows must be aligned at all but one of the $d_i$ where we have at most two choices of alignment. This preserves the congruence conditions on the sequences of consecutive $0$'s and $1$'s of $\mathrm{Ab}(\lambda')$ since both the $1$'s and $0$'s are being inserted in multiples of $|ab|$ within each column. Thus $\lambda'$ satisfies condition $(\ast)$. 

Since the sum of the number of boxes in the $n$-tuple of partitions remains a constant $-rab$, there are only finitely many possible chunks. Consequently, for large enough $n$, every collection of chunks that can be realized into a balanced partition satisfying $(\ast)$ will have been realized. Thus every partition in $B^{-rab}_{1,-1; n + |ab|}$ satisfying condition $(\ast)$ is obtained from a partition in $\lambda \in B^{-rab}_{1,-1;n}$ by choosing where to insert a string of $|ab|$ empty partitions into $Q_n(\lambda)$. Such choices are counted by a sum of binomial coefficients over all realizable chunks which is a polynomial. We have one such polynomial for every residue class modulo $|ab|$ since we can only increase $n$ in multiples of $|ab|$. This is the required quasipolynomial counting partitions of $B^{-rab}_{1,-1;n}$ satisfying $(\ast)$, or equivalently, counting $\# B^r_{a,b;n}$ for $n \gg 0$.  
\end{proof}

\appendix
\section{Cores-and-quotients}
\label{appendix}

We briefly review here the cores-and-quotients bijection for partitions. For more details see for example \cite[Chapter 11]{loehr}. An \textit{abacus} is a function $h : \ZZ \to \{0,1\}$ so that $h(z) = 1$ for $z \ll 0$ and $h(z) =  0$ for $z \gg 0$. We can write this as a sequence of $1$'s and $0$'s consisting of all $1$'s far enough to the left and all $0$'s far enough to the right. To each partition $\lambda$, we can associate an abacus $\mathrm{Ab}(\lambda)$ which encodes the outside edge of the partition by writing a $1$ for each vertical edge segment and writing a $0$ for each horizontal edge segment.

For example, if $\lambda = (4,2,2,1)$ then we construct $\mathrm{Ab}(\lambda)$:

\begin{figure}[h]
\includegraphics[scale=0.4]{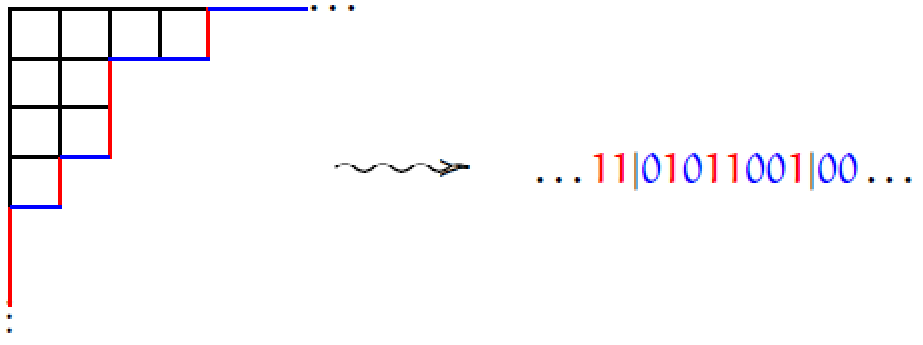}
\end{figure}

\noindent Here we have colored the vertical edge segments red and the horizontal edge segments blue with the corresponding colors for the abacus. This map gives a bijection between partitions and abaci up to translation. Here we have marked off where the edge of the partition begins and ends with a vertical bar. This will always be before the first occurence of $0$ and after the last occurence of $1$. The finite sequence between the bars uniquely determines the abacus and so we will often just work with this finite sequence, filling in $1$'s at the beginning or $0$'s at the end as needed. 

An \textit{abacus with $n$ runners} is an $n$-tuple of abaci that we will picture as $n$ horizontal sequences of $1$'s and $0$'s stacked on top of each other. We will call the $i^{th}$ abacus in this tuple the $i^{th}$ runner. By the above bijection this corresponds to an $n$-tuple of partitions. Let $\mathrm{Ab}(\lambda)$ be the abacus for some partition $\lambda$. Write down the sequence of $1$'s and $0$'s of $\mathrm{Ab}(\lambda)$ vertically in columns of size $n$ starting with the first $0$. This will give an array of $n$ rows of $1$'s and $0$'s which we will interpret as an abacus with $n$ runners that we will denote $\mathrm{Ab}_n(\lambda)$. Equivalently, $\mathrm{Ab}_n(\lambda)$ is an abacus whose $i^{th}$ runner is the subsequence of $\mathrm{Ab}(\lambda)$ of $1$'s and $0$'s in position equal to $i \pmod{n}$. The corresponding $n$-tuple of partitions is the \textit{$n$-quotient} of $\lambda$ which we denote $Q_n(\lambda)$. 

Let us illustrate this with the above example. As we saw, the abacus corresponding to $\lambda = (4,2,2,1)$ is $\ldots 11|\color{red}01011001\color{black}|00 \ldots$. Writing down this sequence vertically in columns of size $n = 3$ gives the following abacus with $3$ runners: 

\begin{align*} 
\ldots &11\color{red}010\color{black}00 \ldots \rightsquigarrow (1) \\
\ldots &11\color{red}111\color{black}00 \ldots \rightsquigarrow \emptyset \\
\ldots &11\color{red}00\color{black}000 \ldots \rightsquigarrow \emptyset 
\end{align*}

\noindent where we have colored the original sequence between the vertical bars in red for illustration. Reading these $3$ abaci accross gives us the $3$-tuple of partitions $Q_3((4,2,2,1)) = ((1),\emptyset,\emptyset)$. 

We construct a new abacus with $n$ runners $\widetilde{\mathrm{Ab}}_n(\lambda)$ from $\mathrm{Ab}_n(\lambda)$ by transposing every occurrence of $01$ so that $\widetilde{\mathrm{Ab}}_n(\lambda)$ corresponds to the $n$-tuple of empty partitions. Finally, reading $\widetilde{\mathrm{Ab}}_n(\lambda)$ vertically from left to right gives a single abacus whose corresponding partition we call the \textit{$n$-core} of $\lambda$, denoted $C_n(\lambda)$. That is, we are undoing the process by which we obtained $\mathrm{Ab}_n(\lambda)$ from $\mathrm{Ab}(\lambda)$. 

Continuing the example from above, $\widetilde{\mathrm{Ab}}_n(\lambda)$ is given by

\begin{align*} 
\ldots &11\color{blue}10\color{black}000 \ldots \\
\ldots &1111100 \ldots \\
\ldots &1100000 \ldots 
\end{align*}

\noindent where we have highlighted in blue the occurence of $01$ that was transposed. Reading this abacus with $n$ runners vertically gives the abacus $\ldots 111001001000\ldots$ which corresponds to the partition $C_3(\lambda) = (4,2)$. 

We need the following well known theorem about partitions:

\begin{thm}\label{thm:cores&quot}(Cores-and-quotients bijection \cite[Theorem 11.22]{loehr}) The map 

$$
\lambda \mapsto (C_n(\lambda),Q_n(\lambda))
$$

\noindent gives a bijection between the set of partitions and pairs of $n$-cores and $n$-quotients. Furthermore, if $Q_n(\lambda) = (\lambda^{(0)},\ldots, \lambda^{(n-1)})$, then 

\begin{equation}\label{cores&quot}
|\lambda| = |C_n(\lambda)| + n \sum_{i = 1}^{n-1} |\lambda^{(i)}|.
\end{equation}

\end{thm}

We also need the following fact which is known though not explicitly stated in the literature (see for example \cite[Theorem 4.5]{nagao}):

\begin{prop}\label{prop:trivialcore} A partition $\lambda$ is $(1,-1;n)$-balanced if and only if it has empty $n$-core. In particular, the cores-and-quotients bijection restricts to a bijection between $B^r_{1,-1;n}$ and the set of $n$-tuples of partitions $(\lambda^{(0)},\ldots, \lambda^{(n-1)})$ satisfying 

$$
\sum_{i = 1}^{n-1} |\lambda^{(i)}| = r.
$$

\end{prop}

\begin{proof} From Theorem 4.5 of \cite{nagao}, there is a commutative diagram 

$$
\xymatrix{
\bigsqcup_r\left(H^r_{1,-1;n}\right)^{S} \ar[r] \ar[d] & \Pi \ar[d] \\
\bigsqcup_r\Hilb^r(H^1_{1,-1;n})^S \ar[r] & \mc{C}_n \times \Pi^n }
$$

\noindent where $S = (\CC^*)^2$ is the torus acting on $\mb{A}^2$, $\Pi$ is the set of all partitions and $\mc{C}_n$ is the set of $n$-cores. The top horizontal map sends a torus fixed point to the corresponding $(1,-1;n)$-balanced partition. The vertical map on the left is a natural bijection induced by an $S$-equivariant diffeomorphism \cite[Lemma 4.1.3]{nagao}

$$
H^r_{1,-1;n} \to \Hilb^r(H^1_{1,-1;n}).
$$

\noindent The vertical map on the right is the cores-and-quotients bijection. 

Now $H^1_{1,-1;n}$ is an $S$-toric variety with $n$ torus fixed points (see Section \ref{sec:toric}). The torus fixed points of $\Hilb^r(H^1_{1,-1;n})$ consist of a choice of monomial ideal supported at each torus fixed point. Consequently $\Hilb^r(H^1_{1,-1;n})^S$ is in bijection with $n$-tuples $(\lambda^{(0)}, \ldots, \lambda^{(n-1)})$ such that 

$$
\sum_{i = 1}^{n-1} |\lambda^{(i)}| = r
$$

\noindent and this is precisely the bottom horizontal map in the diagram. The result then follows from commutativity. 

\end{proof}

\bibliographystyle{amsalpha} 
\bibliography{betti}

\providecommand{\bysame}{\leavevmode\hbox to3em{\hrulefill}\thinspace}
\providecommand{\MR}{\relax\ifhmode\unskip\space\fi MR }
\providecommand{\MRhref}[2]{%
  \href{http://www.ams.org/mathscinet-getitem?mr=#1}{#2}
}
\providecommand{\href}[2]{#2}
\begin{thebibliography}{GZLMH10}

\bibitem[BB73]{bialynicki}
A.~Bia{\l}ynicki-Birula, \emph{Some theorems on actions of algebraic groups},
  Ann. of Math. (2) \textbf{98} (1973), 480--497. \MR{0366940 (51 \#3186)}

\bibitem[BBS13]{BBS}
Kai Behrend, Jim Bryan, and Bal{\'a}zs Szendr{\H{o}}i, \emph{Motivic degree
  zero {D}onaldson-{T}homas invariants}, Invent. Math. \textbf{192} (2013),
  no.~1, 111--160. \MR{3032328}

\bibitem[BF13]{buryak-feigin}
A.~Buryak and B.~L. Feigin, \emph{Generating series of the {P}oincar\'e
  polynomials of quasihomogeneous {H}ilbert schemes}, Symmetries, integrable
  systems and representations, Springer Proc. Math. Stat., vol.~40, Springer,
  Heidelberg, 2013, pp.~15--33. \MR{3077679}

\bibitem[BF14]{bezrufinkel}
Roman Bezrukavnikov and Michael Finkelberg, \emph{Wreath {M}acdonald
  polynomials and the categorical {M}c{K}ay correspondence}, Camb. J. Math.
  \textbf{2} (2014), no.~2, 163--190, With an appendix by Vadim Vologodsky.
  \MR{3295916}

\bibitem[BKR01]{bkr}
Tom Bridgeland, Alastair King, and Miles Reid, \emph{The {M}c{K}ay
  correspondence as an equivalence of derived categories}, J. Amer. Math. Soc.
  \textbf{14} (2001), no.~3, 535--554 (electronic). \MR{1824990 (2002f:14023)}

\bibitem[Bou68]{CST}
Nicolas Bourbaki, \emph{Groupes et alg\`ebras {C}h. {V}}, Hermann, {P}aris,
  1968.

\bibitem[Bri68]{brieskorn}
Egbert Brieskorn, \emph{Rationale {S}ingularit\"aten komplexer {F}l\"achen},
  Invent. Math. \textbf{4} (1967/1968), 336--358. \MR{0222084 (36 \#5136)}

\bibitem[Bri12]{bridgeland}
Tom Bridgeland, \emph{An introduction to motivic {H}all algebras}, Adv. Math.
  \textbf{229} (2012), no.~1, 102--138. \MR{2854172 (2012j:14018)}

\bibitem[Bri13]{brion}
Michel Brion, \emph{Invariant {H}ilbert schemes}, Handbook of moduli. {V}ol.
  {I}, Adv. Lect. Math. (ALM), vol.~24, Int. Press, Somerville, MA, 2013,
  pp.~64--117. \MR{3184162}

\bibitem[CLS11]{cls}
David~A. Cox, John~B. Little, and Henry~K. Schenck, \emph{Toric varieties},
  Graduate Studies in Mathematics, vol. 124, American Mathematical Society,
  Providence, RI, 2011. \MR{2810322 (2012g:14094)}

\bibitem[Dub90]{grobner}
Thomas~W. Dub{\'e}, \emph{The structure of polynomial ideals and {G}r\"obner
  bases}, SIAM J. Comput. \textbf{19} (1990), no.~4, 750--775. \MR{1053942
  (91h:13021)}

\bibitem[ES87]{es1}
Geir Ellingsrud and Stein~Arild Str{\o}mme, \emph{On the homology of the
  {H}ilbert scheme of points in the plane}, Invent. Math. \textbf{87} (1987),
  no.~2, 343--352. \MR{870732 (88c:14008)}

\bibitem[ES88]{es2}
\bysame, \emph{On a cell decomposition of the {H}ilbert scheme of points in the
  plane}, Invent. Math. \textbf{91} (1988), no.~2, 365--370. \MR{922805
  (89f:14007)}

\bibitem[Fog68]{fogarty}
John Fogarty, \emph{Algebraic families on an algebraic surface}, Amer. J. Math
  \textbf{90} (1968), 511--521. \MR{0237496 (38 \#5778)}

\bibitem[Fog71]{fogarty2}
\bysame, \emph{Fixed point schemes}, Bull. Amer. Math. Soc. \textbf{77} (1971),
  203--204. \MR{0269661 (42 \#4556)}

\bibitem[G{\"o}t90]{gottsche}
Lothar G{\"o}ttsche, \emph{The {B}etti numbers of the {H}ilbert scheme of
  points on a smooth projective surface}, Math. Ann. \textbf{286} (1990),
  no.~1-3, 193--207. \MR{1032930 (91h:14007)}

\bibitem[G{\"o}t09]{modularforms}
\bysame, \emph{Invariants of moduli spaces and modular forms}, Rend. Istit.
  Mat. Univ. Trieste \textbf{41} (2009), 55--76 (2010). \MR{2676965
  (2011g:14030)}

\bibitem[GZLMH10]{GZLMH}
S.~M. Gusein-Zade, I.~Luengo, and A.~Melle-Hern{\'a}ndez, \emph{On generating
  series of classes of equivariant {H}ilbert schemes of fat points}, Mosc.
  Math. J. \textbf{10} (2010), no.~3, 593--602, 662. \MR{2732574 (2011m:14006)}

\bibitem[Hai98]{haiman}
Mark Haiman, \emph{{$t,q$}-{C}atalan numbers and the {H}ilbert scheme},
  Discrete Math. \textbf{193} (1998), no.~1-3, 201--224, Selected papers in
  honor of Adriano Garsia (Taormina, 1994). \MR{1661369 (2000k:05264)}

\bibitem[IN96]{in1}
Yukari Ito and Iku Nakamura, \emph{Mc{K}ay correspondence and {H}ilbert
  schemes}, Proc. Japan Acad. Ser. A Math. Sci. \textbf{72} (1996), no.~7,
  135--138. \MR{1420598 (97k:14003)}

\bibitem[Ish02]{ishii}
Akira Ishii, \emph{On the {M}c{K}ay correspondence for a finite small subgroup
  of {${\rm GL}(2,\mathbb{C})$}}, J. Reine Angew. Math. \textbf{549} (2002),
  221--233. \MR{1916656 (2003d:14021)}

\bibitem[Kid01]{kidoh}
Rie Kidoh, \emph{Hilbert schemes and cyclic quotient surface singularities},
  Hokkaido Math. J. \textbf{30} (2001), no.~1, 91--103. \MR{1815001
  (2001k:14009)}

\bibitem[Li]{lili}
Li~Li, \emph{Hilbert schemes of points on a stack}, Unpublished.

\bibitem[Loe11]{loehr}
Nicholas~A. Loehr, \emph{Bijective combinatorics}, Discrete Mathematics and its
  Applications (Boca Raton), CRC Press, Boca Raton, FL, 2011. \MR{2777360
  (2012d:05002)}

\bibitem[MS10]{maclagan-smith}
Diane Maclagan and Gregory~G. Smith, \emph{Smooth and irreducible multigraded
  {H}ilbert schemes}, Adv. Math. \textbf{223} (2010), no.~5, 1608--1631.
  \MR{2592504 (2011e:14009)}

\bibitem[Nag09]{nagao}
Kentaro Nagao, \emph{Quiver varieties and {F}renkel-{K}ac construction}, J.
  Algebra \textbf{321} (2009), no.~12, 3764--3789. \MR{2517812 (2010f:16019)}

\bibitem[Nak97]{nakajima1}
Hiraku Nakajima, \emph{Heisenberg algebra and {H}ilbert schemes of points on
  projective surfaces}, Ann. of Math. (2) \textbf{145} (1997), no.~2, 379--388.
  \MR{1441880 (98h:14006)}

\bibitem[Nak98]{quiverkacmoody}
\bysame, \emph{Quiver varieties and {K}ac-{M}oody algebras}, Duke Math. J.
  \textbf{91} (1998), no.~3, 515--560. \MR{1604167 (99b:17033)}

\bibitem[OS03]{quot}
Martin Olsson and Jason Starr, \emph{Quot functors for {D}eligne-{M}umford
  stacks}, Comm. Algebra \textbf{31} (2003), no.~8, 4069--4096, Special issue
  in honor of Steven L. Kleiman. \MR{2007396 (2004i:14002)}

\bibitem[PP07]{contfrac}
Patrick Popescu-Pampu, \emph{The geometry of continued fractions and the
  topology of surface singularities}, Singularities in geometry and topology
  2004, Adv. Stud. Pure Math., vol.~46, Math. Soc. Japan, Tokyo, 2007,
  pp.~119--195. \MR{2342890 (2008k:32082)}

\bibitem[Rei02]{reid}
Miles Reid, \emph{La correspondance de {M}c{K}ay}, Ast\'erisque (2002),
  no.~276, 53--72, S{\'e}minaire Bourbaki, Vol. 1999/2000. \MR{1886756
  (2003h:14026)}

\bibitem[VW94]{vafawitten}
Cumrun Vafa and Edward Witten, \emph{A strong coupling test of {$S$}-duality},
  Nuclear Phys. B \textbf{431} (1994), no.~1-2, 3--77. \MR{1305096 (95k:81138)}

\bibitem[{Wan}99]{weiqiang}
W.~{Wang}, \emph{{Hilbert schemes, wreath products, and the McKay
  correspondence}}, ArXiv Mathematics e-prints (1999).

\end{thebibliography}

\end{document}